\documentclass[12pt,leqno]{article}
\usepackage{graphicx, fullpage}
\usepackage{amsmath,amssymb,amsthm,amscd, bm}
\usepackage{fancyhdr}
\usepackage[mathscr]{eucal}
\usepackage{amsfonts}
\begin{document}
\parskip=6pt

\theoremstyle{plain}

\newtheorem {thm}{Theorem}[section]
\newtheorem {lem}[thm]{Lemma}
\newtheorem {cor}[thm]{Corollary}
\newtheorem {defn}[thm]{Definition}
\newtheorem {prop}[thm]{Proposition}
\numberwithin{equation}{section}
%\newtheorem{proof}{Proof}
%\newtheorem{proof2}{Proof of Theorem 1.2}
%\newtheorem{proof3}{Proof of Theorem 1.3}
%\newtheorem{proof4}{Proof of Theorem 1.4}
%\numberwithin{proof2}{Proof of Theorem 1.2}
\def\cal{\mathcal}
\newcommand{\cF}{\cal F}
\newcommand{\cA}{\cal A}
\newcommand{\cC}{\cal C}
\newcommand{\cO}{{\cal O}}
\newcommand{\cE}{{\cal E}}
\newcommand{\cU}{{\cal U}}
\newcommand{\cM}{{\cal M}}
\newcommand{\cD}{{\cal D}}
\newcommand{\cK}{{\cal K}}
\newcommand{\bC}{\mathbb C}
\newcommand{\bP}{\mathbb P}
\newcommand{\bN}{\mathbb N}
\newcommand{\bA}{\mathbb A}
\newcommand{\bR}{\mathbb R}
\newcommand{\lf}{\lfloor}
\newcommand{\rf}{\rfloor}
\newcommand{\var}{\varepsilon}
\newcommand{\End}{\text{End}}
\newcommand{\GL}{\text{GL}}
\newcommand{\loc}{\text{loc}}
\newcommand{\dist}{\text{dist}}
\newcommand{\op}{\text{op}}
\newcommand{\Id}{\text{Id}}
\newcommand{\weak}{\text{weak}}
\newcommand{\Hom}{\text{Hom}}
\newcommand{\rk}{\roman{rk }}
\newcommand{\ds}{\displaystyle}
\newcommand{\oOmega}{\overline\Omega}
\renewcommand\qed{ }
\newcommand{\opartial}{\overline\partial}
\begin{titlepage}
\title{\bf Noncommutative potential theory\thanks{Research partially supported by NSF grant DMS--1464150.\newline
2010 Mathematics Subject Classification 31C05, 31C10, 47A56, 47A68}}
\author{L\'aszl\'o Lempert\\ Department of  Mathematics\\
Purdue University\\West Lafayette, IN
47907-2067, USA}
\thispagestyle{empty}
%\maketitle
\end{titlepage}
\date{}
\maketitle
\abstract
We propose to view hermitian metrics on trivial holomorphic vector bundles $E\to\Omega$ as noncommutative analogs of functions defined on the base $\Omega$, and curvature as the notion corresponding to the Laplace operator or $\partial\opartial$.
We discuss noncommutative generalizations of basic results of ordinary potential theory, mean value properties, maximum principle, Harnack inequality, and the solvability of Dirichlet problems.
\endabstract

\section{Introduction}

Traditional potential theory is the study of the Laplace operator, harmonic and subharmonic functions, and related notions.
The Laplacian, while an analytic object, has geometric content as well:\ the curvature of a holomorphic line bundle over a Riemann surface is expressed through it.
By noncommutative potential theory we mean the study of hermitian metrics on holomorphic vector bundles of higher rank, in the spirit of traditional potential theory:\ through maximum principles, averaging properties, the Dirichlet problem, regularization, and more.
Although in complex geometry chiefly vector bundles of finite rank occur, one still encounters there and elsewhere---e.g., in harmonic analysis or mathematical physics---bundles with Hilbert or Banach space fibers, or even more general bundle--like objects, see e.g.~[ADW,B,L3,LSz,Rc].
Accordingly, in this paper we will discuss vector bundles with Hilbert space fibers and hermitian metrics on them.
At the same time, we will focus on trivial Hilbert bundles, typically over open subsets $\Omega$ of $\bC$.
Some of our results clearly have implications for general vector bundles (and higher dimensional bases), but the analogy between the Laplacian or $\partial\opartial$ and curvature in a general vector bundle is clearest if the bundle is locally trivialized first.

We shall write $\Hom(V,W)$ for the space of continuous linear maps between Banach spaces $V$ and $W$, End$V$ for Hom$(V,V)$, and $\GL(V)\subset$ End$V$ for the group of invertible elements.
Let $(V,\langle ,\rangle)$ be a complex Hilbert space and $\Omega\subset\bC^n$ open.
A hermitian metric on the trivial bundle $\Omega\times V\to\Omega$ is a function $h\colon\Omega\times V\times V\to\bC$ given by an operator valued function $P$ on $\Omega$
\begin{equation}
h(z,u,v)=\langle P(z) u,v\rangle,\qquad z\in\Omega,\ u,v\in V.
\end{equation}
For the sake of simplicity in this introduction we restrict ourselves to smooth metrics, meaning that $P=P^*\colon\Omega\to\End\, V$ is a $C^\infty$ map taking values in positive invertible operators; but in the main body of the paper we will deal with rougher metrics as well.
We write $\End^+ V\subset\End\, V$ for the cone of positive invertible operators.

The curvature of $h$ or of $P$ is the $\End \,V$ valued $(1,1)$ form on $\Omega$
\begin{eqnarray}
R=R^h=R^P&=&\opartial (P^{-1}\partial P)=P^{-1}\opartial\partial P-P^{-1}\opartial P\wedge P^{-1}\partial P\\
&=&P^{-1}\Bigl(\sum^n_{\mu,\nu=1} P_{\overline z_\mu} P^{-1} P_{z_\nu}-P_{\overline z_\mu z_\nu})  \Bigr) dz_\nu\wedge d\overline z_\mu.\nonumber
\end{eqnarray}
When $\dim V=1$ and $P$ is multiplication by a positive $p\in C^\infty(\Omega)$, $R=\opartial\partial\log p$, and zero curvature corresponds to (pluri)harmonicity.
Our first result is about solving a noncommutative Dirichlet problem on the disc in $\bC$ for general metrics.

\begin{thm}Suppose $\Omega=\{z\in\bC\colon |z| < 1\}$ and $F\colon\partial\Omega\to\End^+ V$ is continuous for the norm topology on $\End^+V\subset\End\, V$.
Then there is a unique continuous $P\colon\overline\Omega\to\End^+V$ that is smooth over $\Omega$,
\begin{eqnarray*}
P&=&F\text{ on }\ \partial\Omega,\\
R^P&=&0\ \text{ on }\ \Omega.
\end{eqnarray*}
Over $\Omega$ one can write $P=H^*H$ with a holomorphic $H\colon\Omega\to \GL(V)$.
If $F$ takes values in a unital $C^*$--subalgebra $\cA$ of $\End\, V$, then so will $P$, and $H$ can be taken with values in $\cA$.
\end{thm}

Related results, when $\cA=\End V$, have been known for quite a while.
When $\dim V<\infty$, Coifman and Semmes solved Dirichlet problems not only for hermitian but for Finsler metrics as well.
Even earlier, Masani and Wiener solved a Dirichlet problem when $V$ is a finite dimensional Hilbert space, and $F,\log \|F\|$ are only integrable.
Then the conclusion is necessarily weaker than in our theorem.
Devinatz, Douglas, Helson, Foia\c s and Sz.--Nagy subsequently extended this latter result to separable $V$.
This author considered Dirichlet problems with boundary values more regular than continuous, again when $\dim V<\infty$.
See [CS,De,Do,H,L1,SzF,WM].

When the base $\Omega$ is one dimensional, semipositivity/negativity of the curvature simply means
\[
P_{\overline z} P^{-1} P_z-P_{ z\overline z}\geq 0,\text{ respectively }\leq 0,
\]
and we next turn to mean value properties of such metrics.

\begin{thm}A smooth hermitian metric as in (1.1) has seminegative curvature if and only if for every disc $\{z\in\bC\colon |z-a|\leq r\}$ contained in $\Omega$
\begin{equation}
\int_{|z-a|=r} P|dz|-\int_{|z-a|=r} Pdz \left(\int_{|z-a|=r} P|dz|\right)^{-1}\int_{|z-a|=r} Pd\overline z\geq 2\pi r P(a).
\end{equation}
(Here $|dz|$ refers to integration with respect to arc length.) If the curvature is seminegative, the left hand side of (1.4), divided by $r$, is an increasing function of $r$.
\end{thm}

This is clearly analogous to the mean value property of subharmonic functions.
But even when $\dim V=1$ the two are different, for then (1.4) boils down to characterizing subharmonicity of $u=\log p$ in terms of integrals of $p$, rather than of $\log p$.

In Section 4 we will also prove a related characterization of semipositive curvature (but it is {\sl not} (1.3) with the inequality sign reversed).

Behind these results is a maximum principle.
Consider an open $\Omega\subset\bC^n$ and hermitian metrics $h,k$ on trivial holomorphic vector bundles $E=\Omega\times V\to\Omega$, $F=\Omega\times W\to\Omega$.
Write $h_z(u,v)$ for the inner product $h(z,u,v)$ on $V$, for $z\in\Omega$, and similarly $k_z$.
A holomorphic homomorphism $A\colon E\to F$ can be thought of as a holomorphic map $\Omega\to\Hom (V,W)$.
Its norm $\|A\|\colon \Omega\to [0,\infty)$ is obtained by taking for each $z\in\Omega$ the norm of the operator $A(z)\colon (V,h_z)\to (W,k_z)$.

\begin{thm}If $A$ decreases curvature in the sense that for $z\in\Omega,\xi\in T^{1,0}_z\Omega$
\[
{k_z (R^k (\xi,\overline\xi)Av,Av)\over k_z(Av,Av)} \leq {h_z(R^h(\xi,\overline\xi)v,v)\over h_z(v,v)} ,\quad v\in V\text{ such that }A(z)v\neq 0,
\]
then $\log \|A\|$ is plurisubharmonic.
In particular, $\log \|A\|$ satisfies the maximum principle.
\end{thm}

This result is not new, it is a special case of what we proved in [L2].
Related results had been known earlier:\ Coifman and Semmes proved an analogous result for Finsler metrics when $\dim V<\infty$, and Berndtsson proved the infinite rank case when $R^h$ or $R^k=0$, see [BK,CS].

As said, various results in the paper, even if formulated only for bundles over one dimensional bases, have obvious generalizations to higher dimensional bases. But not all these generalizations are satisfactory. For example, an integral characterization of Nakano semipositivity/negativity in the spirit of Theorems 1.2 (or 4.2) and 4.7 is lacking. Yet such a characterization would be useful to study Nakano curvature of uniform limits of, say, Nakano semipositively curved hermitian metrics. 

I am grateful to Kuang--Ru Wu for his questions and critical remarks concerning the first version of this paper.

\section{Smoothness classes of hermitian metrics}

What one means by a hermitian metric of class $C^k$ on a vector bundle of finite rank is unambiguous, but in bundles of infinite rank several definitions are possible depending on the topology one uses on spaces of operators.
Since in matters of smoothness there is no difference between hermitian metrics and general sesquilinear forms, in discussing smoothness classes we will deal with the latter.
We will also allow the base to be a subset of real Euclidean space, and later a smooth manifold.

So, let $\Omega\subset\bR^m$ be open and $(V,\langle \ , \ \rangle)$ a complex Hilbert space.
A sesquilinear form on the bundle $E=\Omega\times V\to\Omega$ is a function $h\colon\Omega\times V\times V\to\bC$ such that 
\[
h_x= h(x,\cdot,\cdot)\colon V\times V\to\bC
\]
is a continuous sesquilinear form for each $x\in\Omega$.
Such an $h$ can be represented as
\begin{equation}
h_x(v,w)=\langle P(x)v,w\rangle,\qquad\text{with }P\colon\Omega\to\End\, V.
\end{equation}
The weakest notion of $C^k$ smoothness, $k=0,1,\ldots,\infty$, is obtained by requiring that $h_x(v,w)=\langle P(x),v,w\rangle$ should be a $C^k$ function of $x\in\Omega$ when $v,w\in V$ are fixed.
If this is so, we say $h$ or $P$ are $C^k_{\text{weak}}$, or that $P\in C^k_{\text{weak}}(\Omega,\End \,V)$.
The strongest requirement is that the map $P\colon\Omega\to\End\, V$ should be $C^k$, when $\End\, V$ is endowed with the operator norm; we then say $h$ or $P$ are $C^k_{\op}$, or that $P\in C_{\op}^k(\Omega,\End\, V)$.
When $k=0$, we just write $C_{\weak}$, $C_{\op}$.

In our context, as elsewhere, H\"older classes $C^k$ with nonintegral $k$ behave better than those with integral $k$.
For one thing, their weak and operator norm versions turn out to coincide.
Because of this the notation will not indicate which topology on $\End\, V$ is used.
Their definition is as follows.

Write $\lfloor k\rfloor$ for the integer part of $k\in\bR$.
If $k\in (0,\infty)$ is nonintegral, we say that $h$ or $P$ is $C^k$, or $P\in C^k(\Omega,\End\, V)$, if for any $v,w\in V$ the function $x\mapsto h_x(v,w)$ has partials of order $\lf k\rf$ on $\Omega$, and these partials are locally H\"older continuous with exponent $k-\lfloor k\rfloor$.
If $k=0,1,\ldots$ and the partials of order $k$ are locally Lipschitz continuous, i.e., H\"older continuous with exponent 1, we say $h$ and $P$ are $C^{k,1}$, or that $P\in C^{k,1}(\Omega,\End\, V)$.

\begin{prop}Let $k\in [0,\infty)$.
If $P\in C^k_{\weak} (\Omega,\End\, V)$ (when $k$ is integral) or $P\in C^k(\Omega,\End\, V)$ (otherwise), and $\varphi,\psi\colon \Omega\to V$ are $C^k$ functions in the norm topology of $V$, then $\langle P\varphi,\psi\rangle\colon\Omega\to\bC$ is also $C^k$.
Furthermore, if $k\in (0,\infty)$ is nonintegral and $P\in C^k (\Omega,\End\, V)$, then $P\in C_{\op}^{\lfloor k\rfloor} (\Omega,\End\, V)$ and its partials of order $\lfloor k\rfloor$ are locally H\"older continuous with exponent $k-\lfloor k\rfloor$.
\end{prop}

Thus, if $Q$ is a partial of $P$ of order $\lf k\rf$, and $K\subset\Omega$ is compact, there is a constant $C$ such that
\[
\|Q(x)-Q(y)\|\leq C|x-y|^{k-\lf k\rf},\quad x,y\in K.
\]
Here $||\,||$  denotes the operator norm on $\End V$; but below we also use it to denote the norm on $V$ induced by the inner product.

\begin{prop}Let $k=0,1,\ldots$.
If $P\in C^{k,1}(\Omega,\End\, V)$ then $P$ is $C^k_{\op}$, and its partials of order $k$ are locally Lipschitz continuous.
\end{prop}

Since $C^k_{\weak}\subset C^{k-1,1}$ for $k\in\bN$, Proposition 2.2 implies

\begin{prop}$C^k_{\weak}(\Omega,\End\, V)\subset C_{\op}^{k-1}(\Omega,\End\, V)$ when $k\in\bN$, and so $C^\infty_{\weak}=C_{\op}^\infty$.
\end{prop}

\begin{proof}We will only prove Proposition 2.1, the proof of Proposition 2.2 is analogous.

(a)\ Suppose $k$ is nonintegral.
It suffices to prove the second statement in Proposition 2.1, that we do by induction on $\lf k\rf$.
Assume first $0<k<1$.
Given sequences $x_j \neq y_j\in\Omega$ with limits $x,y\in\Omega$, for fixed $v,w\in V$ the sequence 
\[
\left\langle {P(x_j)-P(y_j)\over |x_j-y_j|^k}\ v,w\right\rangle,\quad j=1,2,\ldots
\]
is bounded.
Two applications of the principle of uniform boundedness then give that $\|P(x_j)-P(y_j)\|/|x_j-y_j|^k$ is bounded, whence the claim follows.

Next suppose that $l\in\bN$, the proposition has been proved for exponents $<l$, and $P$ is $C^k$ with $k\in (l,l+1)$.
Write $x_\mu$ for the coordinates on $\Omega \ (\mu=1,2,\ldots,m)$, and $\partial_\mu$ for $\partial/\partial x_\mu$.
For fixed $v,w\in V$ representing $\partial_\mu \langle Pv,w\rangle$ as the limit of difference quotients, the principle of uniform boundedness again applies and gives a $Q_\mu\colon\Omega\to\End\, V$ such that $\partial_\mu \langle Pv,w\rangle=\langle Q_\mu v,w\rangle$.
Thus $\langle Q_\mu v,w\rangle$ is $C^{k-1}$, and by the inductive hypothesis $Q_\mu$ has partials of order $l-1$, locally H\"older continuous with exponent $k-1-(l-1)=k-\lf k\rf$.
In particular, $Q_\mu\in C_{\op}(\Omega,\End V)$.
Hence 
\begin{multline*}
 \int_a^b Q_\mu (x_1,\ldots,x_{\mu-1},t,x_{\mu+1},\ldots)dt=\\
 P(x_1,\ldots,x_{\mu-1},b,x_{\mu+1},\ldots)-P(x_1,\ldots,x_{\mu-1},a,x_{\mu+1},\ldots),
\end{multline*}
whenever the path over which we integrate is in $\Omega$.
Thus $Q_\mu =\partial_\mu P$, partial understood in the norm topology.
But then the partials of $P$ of order $l=\lf k\rf$ are locally H\"older continuous with exponent $k-\lf k\rf$, as claimed.

(b)\ Next suppose $P\in C_{\weak}(\Omega,\End\, V)$.
Given a sequence $x_j\in\Omega$ with limit $x\in\Omega$, again by the principle of uniform boundedness $\sup_j\|P(x_j)\|<\infty$.
Hence
\begin{multline*}
 \langle P(x_j)\varphi(x_j),\psi(x_j)\rangle-\langle P(x)\varphi(x),\psi(x)\rangle=\langle (P(x_j)-P(x))\varphi(x),\psi(x)\rangle+\\
 \langle P(x_j)(\varphi(x_j)-\varphi(x)),\psi(x)\rangle+\langle P(x_j)\varphi(x_j), \psi(x_j)-\psi(x)\rangle\to 0
\end{multline*}
as $j\to\infty$, and $\langle P\varphi,\psi\rangle$ is indeed continuous.

(c)\ Finally, suppose $k\in\bN$ and $P$ is $C^k_{\weak}$.
We assume, inductively, that the claim in Proposition 2.1 is true when $k$ is replaced by $k-1$.
As in part (a) of the proof, there are weak partials $Q_\mu\colon\Omega\to\End\, V$ such that $\partial_\mu\langle Pv,w\rangle=\langle Q_\mu v,w\rangle$ for all $v,w\in V$.
If $\varphi,\psi\colon\Omega\to V$ are $C^k$, or even just $C^1$, we claim that $\langle P\varphi,\psi\rangle$ has partial derivatives
\begin{equation}
\partial_\mu\langle P\varphi,\psi\rangle=\langle Q_\mu \varphi,\psi\rangle+\langle P\partial_\mu\varphi,\psi\rangle + \langle P\varphi,\partial_\mu\psi\rangle.
\end{equation}

With $x\in\Omega$ and $\xi=(\xi_1,\ldots,\xi_m)\in\bR^m$ small we can write 
\[
\varphi(x+\xi)=\varphi(x)+\sum^m_{\nu=1} a_\nu (x,\xi)\xi_\nu,\qquad \psi(x+\xi)=\psi(x)+\sum^m_{\nu=1} b_\nu (x,\xi)\xi_\nu,
\]
where $a_\nu$, $b_\nu$ are continuous functions of $x,\xi$.
Then
\begin{multline*}
\langle P(x+\xi)\varphi(x+\xi),\psi(x+\xi)\rangle=\langle P(x+\xi)\varphi(x),\psi(x)\rangle\\
+\sum_\nu\xi_\nu \langle P(x+\xi)\varphi(x),b_\nu (x,\xi)\rangle+\sum_\nu\xi_\nu
\langle P(x+\xi) a_\nu (x,\xi),\psi(x+\xi)\rangle.
\end{multline*}
The first term on the right has $\partial/\partial\xi_\mu$ partials by assumption, equal to $\langle Q_\mu (x+\xi)\varphi(x),\psi(x)\rangle$.
As to the other inner products on the right, by part (b) of this proof they are continuous functions of $x$ and $\xi$.
We conclude that all terms on the right have $\partial/\partial\xi_\mu$ partials at $\xi=0$; these partials add up to
\[
\langle Q_\mu (x)\varphi(x),\psi (x)\rangle+\langle P(x)\varphi(x),\partial_\mu \psi(x)\rangle + \langle P(x)\partial_\mu\varphi (x),\psi(x)\rangle,
\]
which proves (2.2).

Now $Q_\mu$ is $C^{k-1}_{\weak}$.
If $\varphi,\psi$ are $C^k$, by the inductive assumption the right hand side of (2.2) is $C^{k-1}$, and $\langle P\varphi,\psi\rangle$ is indeed $C^k$.
\end{proof}

A variant of Proposition 2.1 concerning upper semicontinuity (u.s.c.) also holds:

\begin{prop}Suppose $P=P^*\colon\Omega\to\End\, V$ is bounded below and is weakly u.s.c~in the sense that $\langle Pv,v\rangle$ is u.s.c.~for $v\in V$.
With any $\varphi\in C(\Omega,V)$ then $\langle P\varphi,\varphi\rangle$ is also u.s.c.
\end{prop}

\begin{proof}Upon adding a constant to $P$, we can arrange that $P\geq 0$.
Let $Q=P^{1/2}$.
Given a sequence $x_j\in\Omega$ with limit $x\in\Omega$, for any $v\in V$ the sequence $\|Q(x_j)v\|^2=\langle P(x_j)v,v\rangle$ is bounded.
Hence $\|Q(x_j)\|$ is bounded, and so is $\|P(x_j)\|$.
As before,
\begin{multline*}
\langle P(x_j)\varphi(x_j),\varphi(x_j)\rangle-\langle P(x)\varphi(x),\varphi(x)\rangle=\langle (P(x_j)-P(x))\varphi(x),\varphi(x)\rangle\\
+\langle P(x_j)(\varphi(x_j)-\varphi(x)),\varphi(x)\rangle+\langle P(x_j)\varphi(x_j),\varphi(x_j)-\varphi(x)\rangle,
\end{multline*}
whence
\[
\limsup_{j\to\infty} \langle P(x_j)\varphi(x_j),\varphi(x_j)\rangle\leq \langle P(x)\varphi(x),\varphi(x)\rangle,
\]
as claimed.
\end{proof}

We will also need the following result.
Let $(W,\ \|\ \|)$ be a Banach space.

\begin{prop}Fix $k\in\bN$ and consider a sequence $P_j\colon\Omega\to W$ of functions, $C^k$ in the norm topology.
Assume that $P_j(x)$ converges in norm, uniformly for $x\in\Omega$, and that the partials of $P_j$, of order $\leq k$, are uniformly equicontinuous on $\Omega$.
Then these partials converge in norm, locally uniformly on $\Omega$.
\end{prop}

\begin{proof}First let $k=1$; say, we want to prove that $\partial_1 P_j$ converges locally uniformly.
Given a compact $K\subset\Omega$ and $\var>0$, choose $0<\delta<\dist (K,\partial\Omega)$ so that $\|\partial_1 P_j(x)-\partial_1 P_j(y)\|<\var/4$ when $j\in\bN$, $x\in K$, and $|x-y|\leq\delta$.
When $x=(x_1,\ldots,x_m)\in K$
\begin{multline}
\Big\|{P_j(x_1+\delta,x_2,\ldots,x_m)-P_j(x)\over \delta}-\partial_1 P_j(x)\Big\|=\\
{1\over\delta}\Big\| \int_0^\delta \big(\partial_1 P_j (x_1+t,x_2,\ldots,x_m)-\partial_1 P_j(x)\big)dt\Big\| < {\var\over 4}.
\end{multline}
Next choose $j_0$ so that for $i,j>j_0$ and $y\in\Omega$
\[
\|P_i(y)-P_j(y)\|<\delta\var/4.
\]
By (2.3), if $x\in K$, $\xi=(x_1+\delta,x_2,\ldots,x_n)$, and $i,j>j_0$ 
\[
\|\partial_1 P_i(x)-\partial_1 P_j(x)\|<\delta^{-1} (\|P_i(\xi)-P_j(\xi)\|+\|P_i(x)-P_j(x)\|)+\var/2<\var.
\]
Therefore $\partial_1 P_j$ indeed converges locally uniformly, as do all other first partials.

The case $k>1$ follows by induction.
\end{proof}

All the smoothness classes that we have introduced are invariant under $C^\infty$ diffeomorphisms $\Omega\to\Omega'$.
For this reason they make sense over differential manifolds $M$ as well.
The corresponding spaces of functions will be denoted $C^k_{\weak} (M,\End\, V)$, etc.
They all have variants $C^k(\overline M,\End\, V)$, etc. with $\overline M$ a compact manifold with boundary.
For example, when $k\in (0,\infty)$ is nonintegral, we fix a finite open cover $\overline M=U_1\cup\ldots\cup U_s$ such that each $\overline U_j$ is contained in a  coordinate chart, and let $C^k(\overline M,\End\, V)$ consist of continuous $P\colon\overline M\to\End\, V$ whose restrictions $P|U_j\cap\text{ int} M$ have partials of order $\lf k\rf$, partials $Q$ that satisfy
\begin{equation}
\sup \left\{ {\|Q(x)-Q(y)\|\over |x-y|^{k-\lf k\rf}}\quad\colon\quad x,y\in U_j\cap\text{ int }M\right\}<\infty
\end{equation}
for all $j$. The partials, $x-y$, and its length $|x-y|$ are computed using the coordinates on $\overline U_j$.
This space turns out to be a Banach space if the norm of $P$ is defined as the maximum of $\sup_{x\in\overline M} \|P(x)\|$ and the quantities appearing in (2.4), for all choices $j=1,\ldots,s$ and partial derivatives $Q$. Multiplication in $C^k(\overline M,\End\, V)$ is then continuous, and by rescaling the norm we can turn $C^k(\overline M,\End\, V)$ into  a Banach algebra.

When $S\subset\End\, V$, we write $C^k(M,S)$ etc.~for the space of $P\in C^k(M,\End\, V)$ etc.~that take values in $S$.

\section{A Dirichlet problem}

Given a complex Hilbert space $V$, consider a unital $C^*$ subalgebra $\cA\in\End\, V$.
The most important case is $\cA=\End\, V$.
Denote the unit in $\cA$ by $\mathbf 1$.
We write $\cA^\times \subset\cA$ for the group of invertible elements and $\cA^+\subset\cA^\times$ for the cone of (self adjoint) positive elements.
This latter is not completely consistent with usage in $C^*$ algebra theory, where $\cA^+$ would contain not necessarily invertible elements as well. An equivalent definition of our $\cA^+$ would be the set of self adjoint $S\in\cA$ that satisfy $S\ge\var\Id$ with some $\var>0$ (cf. [C, p. 243, Exercise 8]). 
In this section $\Omega=\{z\in\bC\colon |z| < 1\}$.

\begin{thm}Given $F\in C_{\op}(\partial\Omega,\cA^+)$, there is a holomorphic $H\colon\Omega\to\cA^\times$ such that the function
\[
P=\begin{cases}
H^*H&\text{ on $\Omega$}\\
F&\text{on $\partial\Omega$}\end{cases}
\]
is in $C_{\op} (\overline\Omega,\cA^+)$.
If a holomorphic $K\colon\Omega\to\cA^\times$ also solves the same problem, then $K=UH$ with a unitary $U\in\cA$.
\end{thm}

From this the existence part of Theorem 1.1 follows because on $\Omega$
\begin{equation}
R^P=\overline\partial (P^{-1} \partial P)=\overline\partial (H^{-1} \partial H)=0.
\end{equation}
In the theorem $H$ itself need not extend continuously to $\overline\Omega$, not even when $\cA=\bC$.
A continuous $f\colon\partial\Omega\to\bR$ whose conjugate function is discontinuous provides a counterexample $F=e^f$.

To prepare the proof we start with the following simple consequence of the maximum principle, our Theorem 1.3 (cf.~[L2, Theorem 1.1 or Theorem 2.4]).

\begin{lem}(a)\ Suppose $P, Q\in C_{\op} (\overline\Omega,\cA^+)$ are in $C_{\op}^2$ over $\Omega$, and their curvature there is 0.
If $P\geq Q$ on $\partial\Omega$, then the same holds on $\Omega$.

(b)\ Suppose $H,K\colon\Omega\to\cA^\times$ are holomorphic, and $H^*H$, $K^*K$ extend to mappings in $C_{\op}(\overline\Omega,\cA^+)$.
If $H^*H\geq K^*K$ on $\partial\Omega$, then the same holds on $\Omega$.
\end{lem}

\begin{proof}(a)\ Let $h,k$ be the hermitian metrics on $\Omega\times V\to\Omega$ determined by $P,Q$.
By Theorem 1.3 the norm of the identity map $A(z)=\Id\colon (V,h_z)\to (V,k_z)$ is logarithmically subharmonic.
Since $\|A(z)\|\leq 1$ when $z\in\partial\Omega$, the maximum principle for subharmonic functions implies $\|A(z)\|\leq 1$ for $z\in\Omega$ as well.

(b)\ This follows by setting $P=H^*H$, $Q=K^*K$, both of which have zero curvature, see (3.1).
\end{proof}

\begin{cor}For $j\in\bN$ let $H_j\colon\Omega\to\cA^\times$ be holomorphic.
Suppose that $H_j^*H_j$ extend to mappings in $C_{\op}(\overline\Omega,\cA^+)$.
If $H_j^* H_j|\partial\Omega$ converge in $C_{\op}(\partial\Omega,\cA^+)$, then $H_j^*H_j$ converge in $C_{\op}(\overline\Omega,\cA^+)$.
\end{cor}

\begin{proof}There are positive numbers $a,b$ such that $a\,\Id\le H_j^*H_j\le b\,\Id$  on $\partial\Omega$ for all $j$, whence by Lemma 3.2b also on $\Omega$. In particular, this implies that for any $\var>0$ there is a $j_0$ such that for $i,j>j_0$ 
\[
(1-\var) H_i^* H_i\leq H_j^* H_j\leq (1+\var) H_i^* H_i\qquad\text{ on }\partial\Omega.
\]
By Lemma 3.2b again, the same holds on $\overline\Omega$, whence the claim.
\end{proof}

Next we prove a perturbative result in the spirit of Theorem 3.1.

\begin{lem}Suppose $k\in (0,\infty)$ is not an integer.
Any constant map $F_0\colon\partial\Omega\to\cA^+$ has a neighborhood $\cU\subset C^k (\partial\Omega,\cA^+)$ such that whenever $F\in\cU$, there is an $H\in C^k (\overline\Omega,\cA^\times)$ that is holomorphic on $\Omega$ and satisfies $H^*H=F$ on $\partial\Omega$.
\end{lem}

\begin{proof}It suffices to prove when $F_0\equiv\mathbf 1$.
Consider
\begin{eqnarray*}
A&=&\{h\in C^k (\overline\Omega,\cA)\colon h|\Omega\text{ is holomorphic}\},\\
B&=&\{f\in C^k(\partial\Omega,\cA)\colon f=f^*\},
\end{eqnarray*}
and let $A_0\subset A$, $B_0\subset B$ consist of maps that vanish at $1\in\partial\Omega$.
With their $C^k$ norms these are Banach spaces. As indicated in Section 2, the $C^k$ norms on $A$, suitably normalized, turn $A$ into a Banach algebra; $A_0$ is a  subalgebra.
We claim that the smooth map
\begin{equation}
A_0\ni h\mapsto (\mathbf 1+h^*) (\mathbf 1+h)|\partial\Omega-\mathbf 1=(h^*+h+h^*h)|\partial\Omega\in B_0
\end{equation}
restricts to a diffeomorphism between certain neighborhoods of $0\in A_0$ and $0\in B_0$.

To justify the claim note that the linearization of (3.2) at $h=0$ is the map
\begin{equation}
A_0\ni h\mapsto (h^*+h)|\partial\Omega\in B_0.
\end{equation}
This is clearly injective:\ \ if the harmonic function $h^*+h$ vanishes on $\partial\Omega$, then it vanishes on all of $\overline\Omega$, whence $h=-h^*$ is both holomorphic and antiholomorphic on $\Omega$.
Therefore $h$ is constant, $h\equiv h(1)=0$.
(3.3) is also surjective for the following reason.
Let $f\in B_0$.
Schwarz's formula
\[
s(z)={1\over 4\pi}\int_0^{2\pi}\ {e^{it}+z\over e^{it}-z}\ f(e^{it}),\quad z\in\Omega,
\]
defines a holomorphic function $s\colon\Omega\to \cA$, that by Privalov's theorem, see [P], extends to a function in $C^k(\overline\Omega,\cA)$.
Thus
\[
s(z)^*+s(z)={1\over 2\pi}\ \int_0^{2\pi}\text{ Re }{e^{it}+z\over e^{it}-z}\ f(e^{it}) dt.
\]
The kernel in the integral being Poisson's, we see that the extension of $s$ to $\partial\Omega$ satisfies $s^*+s=f$.
It follows that $h=s-s(1)\in A_0$ also satisfies $(h^*+h)|\partial\Omega=f$.

So (3.3) is an isomorphism, and by the implicit function theorem (3.2) is indeed a diffeomorphism between neighborhoods of $0\in A_0$ and $0\in B_0$.
Hence, if $F\in B$ is sufficiently close to $F_0\equiv\mathbf 1$, i.e.~if $f=F(1)^{-1/2}FF(1)^{-1/2}-\mathbf 1\in B_0$ is sufficiently small, there is an $h\in A_0$ such that $H=(\mathbf 1+h)F(1)^{1/2}\in A$ satisfies $H^*H|\partial\Omega=F$.
Since $H$ maps into $\cA^\times$ if $h$ is small, the lemma is proved.
\end{proof}

Based on this lemma we will prove Theorem 3.1 by the continuity method.
We can avoid tricky a priori estimates by first proving a generalization of a theorem of Fej\'er and Riesz.

\begin{lem}If $F\in C_{\op}(\partial\Omega,\cA^+)$ is a Laurent polynomial 
\begin{equation}
F(z)=\sum^N_{n=-N}\ F_n z^n,\qquad F_n\in\cA,
\end{equation}
then there are $H_n\in\cA$, $0\leq n\leq N$ such that
\[
H(z)=\sum^N_{n=0} H_n z^n
\]
takes invertible values for $z\in\oOmega$ and satisfies $H^*H=F$ on $\partial\Omega$.
It can be arranged that $H_0\in\cA^+$.
\end{lem}

Helson [He, p.~127] proposes a theorem that, when $V$ is separable and $\cA=\End\, V$, would be the same as our lemma, if one could show that Helson's ``outer factor'' is a function valued in $\End\, V$, not in some other space $\Hom(V,W)$.
Subsequently Rosenblum in [Rs] proved
Lemma 3.5, again when $V$ is separable and $\cA=\End\, V$; in his version $F$ need not even take invertible values.
Both Helson and Rosenblum first prove a general factorization theorem from which they derive the polynomial case.
By contrast, in our approach the polynomial factorization comes first.

We will use the following simple fact.

\begin{lem}Suppose that a sequence of invertible $C_k\in\End\, V$ converges in norm to $C$. Then the following are equivalent:

(a) $C$ is invertible;

(b) $\sup_k ||C_k^{-1}||=s<\infty$;

(c) $C^*C$ is invertible.
\end{lem}
\begin{proof} (a) $\Rightarrow$ (b): There is a positive number $c$ such that $||Cv||\ge c||v||$ for $v\in V$. Hence there is a $k_0$ such that $||C_kv||\ge c||v||/2$ when $k>k_0$, i.e., $||C_k^{-1}||\le 2/c$, and indeed $s<\infty$.

(b) $\Rightarrow$ (a): Since
\[
||C_k^{-1}-C_l^{-1}||=||C_k^{-1}(C_l-C_k)C_l^{-1}||\le s^2||C_l-C_k||\to 0 \qquad\text{as }k,l\to\infty,
\]
$D=\lim_k C_k^{-1}$ exists and satisfies $CD=DC=\Id$.

The equivalence (c) $\Leftrightarrow$ (b) is just (a) $\Leftrightarrow$ (b) applied with $C_k'=C_k^*C_k$.
\end{proof}

\begin{proof}[Proof of Lemma 3.5]Fix a nonintegral $k\in (0,\infty)$.
We will use the spaces $A_0\subset A$, $B_0\subset B$ introduced in the proof of Lemma 3.4.
It suffices to prove when $F(1)=\mathbf 1$, for a general $F$ can be normalized to $F(1)^{-1/2} FF(1)^{-1/2}$.
With such $F$ let $\Phi_t=(1-t)\mathbf 1+tF$, with $t\in [0,1]$, and consider the set $T\subset [0,1]$ of those $t$ for which $\Phi_t$ can be written $H^*H|\partial\Omega$, where $H\in C^k(\oOmega,\cA^\times)$ is holomorphic on $\Omega$.

One can modify the requirement on $H$ without affecting $T$.
For example one can require that $H(z)=\sum_0^N H_n z^n$ be a polynomial.
This in fact is automatic for the following reason.
Let $H(z)=\sum_0^\infty H_n z^n$ and $H(z)^{-1}=\sum_0^\infty K_n z^n$ for $z\in\Omega$, and
\begin{equation}
\Phi_t (z)=\sum^N_{-N} G_n z^n\qquad\text{ for }z\in\partial\Omega.
\end{equation}
If we use (3.5) to extend $\Phi_t(z)$ to all $z\in\bC\backslash \{0\}$, as $r\nearrow 1$ we have 
\[
H^*(rz) H(rz)-\Phi_t (rz)\to 0,\qquad\text{ or }\qquad H^*(rz)-\Phi_t (rz) H(rz)^{-1}\to 0
\]
in $\cA$, uniformly for $z\in\partial\Omega$.
The second limit translates to
\[
\sum^\infty_0 H_n^* r^n z^{-n} -\sum^N_{-N} G_n r^n z^n\sum_0^\infty K_n r^n z^n\to 0 ,\quad r\nearrow 1.
\]
The second term here contains no power $z^m$ with $m<-N$, hence the same holds for the first sum, i.e., $H_n=0$ when $n>N$, and indeed $H(z)=\sum_0^N H_n z^n$.

One can also add the requirement that $H(0)\in\cA^+$, because from a general $H$ one can pass to $UH$, where $U=(H(0)^* H(0))^{1/2} H(0)^{-1}\in\cA^\times$.

Now, $T\subset [0,1]$ is open.
Indeed, if $t\in T$ and $H$ is as in the definition of $T$,
\[
B\ni H^{*-1} \Phi_s H^{-1}|\partial\Omega-\mathbf 1\to 0\quad\text{ as }s\to t.
\]
Hence for $s$ close to $t$, by Lemma 3.4 we can write 
\[
H^{*-1}\Phi_s H^{-1}=K^* K\quad\text{ on }\partial\Omega,
\]
with $K\in C^k(\oOmega,\cA^\times)$, holomorphic on $\Omega$.
Since $\Phi_s=(KH)^*(KH)$, such $s$ is indeed in $T$.

But $T$ is also closed.
For suppose $t_\nu\in T$ and $\lim t_\nu=t$.
As seen above, the corresponding factorization of $\Phi_{t_\nu}$ will be 
\[
\Phi_{t_\nu}=H^{(\nu)*} H^{(\nu)}|\partial\Omega,\qquad H^{(\nu)}(z)=\sum_0^N H_{\nu n} z^n,
\]
and we can assume $H^{(\nu)}(0)=H_{\nu 0}\in\cA^+$.
By Corollary 3.3,
\[
H^{(\nu)}(z)^* H^{(\nu)}(z)=\sum^N_{n,m=0} H_{\nu n}^* H_{\nu m} \overline z^n z^m
\]
converge for every $z\in\oOmega$ as $\nu\to\infty$.
The limits are in $\cA^+$.
Since the functions $\overline z^n z^m$ are independent over $\bC$, this implies $H_{\nu n}^* H_{\nu m}\in\cA$ converge for every $n,m$.
In particular, $\lim_\nu H_{\nu 0}=\lim_\nu (H_{\nu 0}^* H_{\nu 0})^{1/2}=H_0\in\cA^+$ exists.
Therefore
\[
\lim_\nu H_{\nu n}=\lim_\nu H_{\nu 0}^{*-1} (H_{\nu 0}^* H_{\nu n})=H_n\in\cA
\]
also exists for $n=0,\ldots,N$; convergence is in operator norm.
Thus $H(z)=\sum_0^N H_n z^n$ satisfies
\[
H(z)^*H(z)\begin{cases}=\Phi_t(z),&\text{$z\in\partial\Omega$}\\
\in\cA,&\text{$z\in\Omega$}.\end{cases}
\]
But $H(z)^*H(z)$ is invertible for $z\in\oOmega$, hence by Lemma 3.6 so is $H(z)$. This shows $t\in T$, and $T$ is closed.

What we have proved about $T$ implies $T=[0,1]$, in particular $F=\Phi_1$ has the required factorization.
\end{proof}

\begin{proof}[ Proof of Theorem 3.1] To construct $H$ choose a sequence $F_\nu\colon\partial\Omega\to \cA^+$ of Laurent polynomials converging uniformly to $F$, and construct factorizations $H_\nu^* H_\nu|\partial\Omega=F_\nu$ as in Lemma 3.5, making sure that $H_\nu(0)\in\cA^+$.
By Corollary 3.3 $H_\nu^* H_\nu$ converge uniformly on $\overline\Omega$ to some $P\in C_{\op}(\oOmega,\cA^+)$, and $P|\partial\Omega=F$.
In particular the norms $\|H_\nu(z)\|$ are uniformly bounded, say $\|H_\nu(z)\|\leq C$ for $z\in\oOmega$ and $\nu\in\bN$.
By Cauchy's estimate
\begin{equation}
\big\| {\partial^k H_\nu(z)\over\partial z^k}\big\| \leq {Ck\text{!}\over (1-|z|)^k},\quad k=0,1,\ldots,\quad z\in\Omega.
\end{equation}
This in turn implies that the partials of $P_\nu=H_\nu^*H_\nu$, of any fixed order, are locally uniformly bounded on $\Omega$.
By Lemma 2.5 these partials converge locally uniformly as $\nu\to\infty$.
In particular
\begin{equation}
\lim_\nu P_\nu(0)^{-1/2} {\partial^k P_\nu\over\partial z^k}(0)=\lim_\nu \ {\partial^k H_\nu\over\partial z^k} (0)=A_k\in\cA
\end{equation}
exists.
In view of (3.6), $\|A_k\|\leq Ck$!, so that
\[
H(z)=\sum^\infty_{k=0} A_k z^k/k\text{!}
\]
defines a holomorphic $H\colon\Omega\to\cA$.
Now (3.6) shows that for $|z|\leq r<1$ the series
\[
\sum^\infty_{k=0}\ {\partial^k H_\nu\over\partial z^k} (0) {z^k\over k!}\quad (=H_\nu (z))
\]
is termwise dominated in norm by $\sum Cr^k$.
By virtue of (3.7) this implies $H_\nu\to H$ locally uniformly on $\Omega$, whence $H^*H=\lim_\nu H_\nu^* H_\nu=\lim_\nu P_\nu=P$ on $\Omega$; and by  Lemma 3.6  $H$ takes invertible values, as needed.

To show $H$ is unique up to a unitary factor, consider another solution $K$ of the same problem, and
\[
Q=\begin{cases}K^*K&\text{ on }\Omega\\ F&\text{ on }\partial\Omega\end{cases}\in C_{\op}(\oOmega,\cA^+).
\]
Since $R^P=R^Q=0$, Lemma 3.2 implies $P=Q$, and so
\[
K^*\ {\partial^k K\over\partial z^k}={\partial^k Q\over  \partial z^k}={\partial^k P\over\partial z^k}=H^*\ {\partial^k H\over\partial z_k}.
\]
Substituting $z=0$ here, Taylor's formula gives that $K(z)=UH(z)$, where $U=K^*(0)^{-1} H^*(0)$ is unitary because 
\[
K^*(0)^{-1} H^*(0) H(0) K(0)^{-1}=K^*(0)^{-1} P(0) K(0)^{-1}=K^* (0)^{-1} Q(0) K(0)^{-1}=\mathbf 1.
\]
This completes the proof.
\end{proof}

\begin{proof}[ Proof of Theorem 1.1]As we already noted, $P$ constructed in Theorem 3.1 solves the Dirichlet problem by (3.1).
Uniqueness of $P$ follows from Lemma 3.2.
\end{proof}

We finish the subject of Dirichlet problems by proving a regularity result:

\begin{thm}If $k\in (0,\infty)$ is nonintegral and $F\in C^k(\partial\Omega,\cA^+)$, then $H\colon\Omega\to \cA^\times$ in Theorem 3.1 extends to a function in $C^k(\oOmega,\cA^\times)$.
\end{thm}

\begin{proof}Given $\zeta\in\partial\Omega$, we will show that $H$ extends to a $C^k$ function in a neighborhood of $\zeta$.
The automorphisms of $\oOmega$
\[
\varphi_t(z)={z+t\zeta\over 1+t\overline\zeta z}\ ,\quad t\in [0,1),
\]
as $t\to 1$ converge in the $C^\infty$ topology to the constant map $\equiv\zeta$ on any closed arc $I\subset\partial\Omega\backslash \{-\zeta\}$.
Hence $F\circ\varphi_t\to F_0\equiv F(\zeta)$ in $C^k(I,\cA)$, and so by Lemma 3.4 for 
some $t$ there is a $K\in C^k(\oOmega,\cA^\times)$, holomorphic on $\Omega$, such that $K^*K=F\circ\varphi_t$ in a neighborhood of $\zeta\in\partial\Omega$.
Therefore on some closed arc $J\subset\partial\Omega$ containing $\zeta$ 
\[
L^*L=F,\qquad\text{where }L=K\circ\varphi_t^{-1}\in C^k(\oOmega,\cA^\times).
\]
In particular, $(L^*L)(rz)-(H^*H)(rz)\to 0$ as $r\nearrow 1$, uniformly for $z\in J$, or
\[
(H^{*-1} L^* LH^{-1})(rz)\to\mathbf 1\qquad\text{ as }r\nearrow 1,
\]
uniformly for $z\in J$ (since $H,H^{-1}$ are bounded).
Fix $v,w\in V$.
The function $\langle HL^{-1}v,w\rangle$ is bounded and holomorphic on $\Omega$; its almost everywhere existing boundary values on $J$ satisfy
\begin{multline*}
\lim_{r\nearrow 1} \langle (HL^{-1}(rz)v,w\rangle=\\
\lim_{r\nearrow 1} \ \langle(\mathbf 1 -H^{*-1} L^* LH^{-1})(HL^{-1})(rz)v,w\rangle+
\langle (H^{*-1} L^*)(rz)v,w\rangle=\\
\lim_{r\nearrow 1}\langle (H^{*-1} L^*)(rz)v,w\rangle.
\end{multline*}
Thus the bounded holomorphic and antiholomorphic functions $\langle HL^{-1}v,w\rangle$ and\newline
$\langle H^{*-1} L^*v,w\rangle$ share the same boundary values on $J$.
This implies that the former continues analytically to $\bC\backslash\oOmega$ across int $J$. But this, in turn, implies that $HL^{-1}$ continues analytically across $\zeta$, as follows e.g. from the more general [L4, Lemma 6.1].
Hence $H$ itself extends $C^k$ near $\zeta$.
\end{proof}

\section{Mean values}

In this section we will characterize semipositivity/negativity of curvature through mean value properties.
This can be done in generality greater than smooth or even continuous metrics that we have worked with so far.
What the appropriate generality should be is suggested by the case of line bundles.
On a trivial line bundle hermitian metrics of, say, seminegative curvature are determined by plurisubharmonic functions $u$, and it is well understood that it is profitable to allow $u$ to be upper semicontinuous and take $-\infty$ as value.
This latter translates to allowing the metric to assign zero length to nonzero vectors in the fibers.

Accordingly, for a Hilbert space $(V,\langle ,\rangle)$ let
\[
\End^{\geq 0} V=\{A\in\End\, V\colon A=A^*\geq 0\}.
\]
Given an open $\Omega\subset\bC$ and a trivial Hilbert bundle $E\colon\Omega\times V\to\Omega$, by a finite hermitian metric on $E$ we mean a function $h\colon\Omega\times V\times V\to\bC$ that can be written
\begin{equation}
h(z,v,w)=\langle P(z)v,w\rangle,\text{ with }P\colon\Omega\to\End^{\geq 0} V.
\end{equation}
Thus here we will deal with one dimensional bases only.
Mean value properties when $\dim\Omega>1$ follow upon restricting to one dimensional slices.

\begin{defn}We say that a finite hermitian metric $h$, or the corresponding $P$ in (4.1), has seminegative curvature if $\langle P\varphi,\varphi\rangle$ is subharmonic for any holomorphic $\varphi\colon\Omega\to V$.
\end{defn}

This definition has been in use for a while now in various degrees of generality.
It implies, in particular, that $P$ is weakly u.s.c., i.e.~$\langle Pv,v\rangle$ is u.s.c.~for $v\in V$.
It also implies the seemingly stronger condition that $\log\langle P\varphi,\varphi\rangle\colon\Omega\to [-\infty,\infty)$ is subharmonic (because with any holomorphic $f\colon\Omega\to \bC$ the function $\langle Pe^f\varphi,e^f\varphi\rangle$ satisfies the maximum principle, whence $2\text{ Re }f+\log\langle P\varphi,\varphi\rangle$ also satisfies it).
At this point we are allowing subharmonic functions to be $\equiv -\infty$.---When $P\in C^2_{\weak}(\Omega,\End^+ V)$, and its curvature $R^P$ can be computed by (1.2), our current notion of seminegative curvature is equivalent to $P_{z\overline z}\geq P_{\overline z} P^{-1} P_z$.

In the following, integrals of $P$ will be understood in the weak sense:\ \ $\int P$ (over whatever set) is the operator $Q$ that satisfies $\int\langle Pv,w\rangle=\langle Qv,w\rangle$ for all $v,w\in V$.
It suffices to require the latter equality for $v=w$ only, from which the general case follows by polarization.

We write $D_r(a)$ for the disc $\{z\in\bC\colon |z-a| < r\}$.

\begin{thm}For a weakly u.s.c.~$P\colon\Omega\to\End^{\geq 0} V$ the following are equivalent:

(i)\ $P$ has seminegative curvature.

(ii)\ For any disc $\overline{D_r(z_0)}\subset\Omega$, writing $\oint$ for the average $(1/2\pi r)\int_{\partial D_r (z_0)}$, 
\begin{equation}
\begin{pmatrix}
\oint P(z) |dz| & \oint P(z) dz\\
\oint P(z) d\overline z & \oint P(z) |dz|
\end{pmatrix} \geq \begin{pmatrix} P(z_0)&0\\ 0&0\end{pmatrix}
\end{equation}
in $\End\,(V\oplus V)$.

(iii)\ For any disc $\overline{D_r(z_0)}\subset\Omega$ and $t\in (0,\infty)$
\begin{equation}
\oint P(z) |dz|-\oint P(z) dz \left(t\Id_V +\oint P(z) |dz|\right)^{-1}\oint P(z) d\overline z\geq P(z_0).
\end{equation}
\end{thm}

If $\oint P|dz|$ is invertible, then (4.3) is equivalent to the simpler estimate
\begin{equation}
\oint P(z) |dz|-\oint P(z) dz \left(\oint P(z)|dz|\right)^{-1}\oint P(z) d\overline z\geq P(z_0).
\end{equation}
Even for noninvertible $\oint P|dz|$ (4.3) and (4.4) will be equivalent if we define the product in (4.4) as the monotone limit, in the weak or strong operator topology,
\[
\lim_{t\searrow 0} \oint P dz \left(t\Id_V+\oint P |dz|\right)^{-1} \oint P d\overline z.
\]
The equivalence of (4.2), (4.3) is an instance of the following.

\begin{lem}Let $V,W$ be Hilbert spaces, $A=A^*\in\End \,V$, $B\in\Hom (W,V)$, $C=C^*\in\End^{\geq 0} W$.
Then
\begin{equation}
\begin{pmatrix}A&B\\ B^*&C\end{pmatrix}\geq 0
\end{equation}
in $\End\,(V\oplus W)$ if and only if $A\geq B(t\Id_W+C)^{-1} B^*$ for all $t\in (0,\infty)$, which in turn is equivalent to
\[
D=\lim_{t\searrow 0} B(t\Id_W+C)^{-1} B^*
\]
existing in the strong operator topology and satisfying $A\geq D$.
\end{lem}

The result is well known and much used in matrix theory, where $A-D$ is called the Schur complement of $C$, see [H, Chapter 7].
The proof for operators is the same as for matrices:\ \ If in (4.5) we replace $C$ by $C_t=C+t\Id_W$, the resulting inequality for all $t>0$ will be equivalent to the original (4.5).
But
\begin{equation}
\begin{pmatrix}\Id_V&-BC_t^{-1}\\ 0&\Id_W\end{pmatrix}
\begin{pmatrix}A&B\\ B^*&C_t\end{pmatrix}
\begin{pmatrix}\Id_V&0\\ -C_t^{-1} B^*&\Id_W\end{pmatrix}=
\begin{pmatrix}A-B C_t^{-1} B^*&0\\ 0&C_t\end{pmatrix},
\end{equation}
so that (4.5) is equivalent to
\begin{equation}
A-B C_t^{-1} B^*\geq 0\quad\text{ for all }t>0.
\end{equation}
Next suppose that (4.7) holds.
Then $BC_t^{-1}B^*$ is a decreasing function of $t>0$, bounded above by $A$, hence by Vigier's theorem [RSz, p.~261] it converges in the strong operator topology to a $D\leq A$, as $t\searrow 0$.
Conversely, suppose $D=\lim_{t\searrow 0} BC_t^{-1} B^*$ exists and satisfies $D\leq A$.
As the limit is monotone, $BC_t^{-1} B^*\leq A$ for $t>0$.

\begin{proof}[Proof of Theorem 4.2]By Lemma 4.3 (ii) and (iii) are equivalent.
To show (i) $\Rightarrow$ (ii), take $z_0=0$ for simplicity.
Given $u,v\in V$, the function $\varphi(z)=u+iv z/r$ is holomorphic, so that $\langle P\varphi,\varphi\rangle$ is subharmonic.
Noting that on the circle $|z|=r$ we have $dz=iz |dz|/r$, $d\overline z=-i\overline z|dz|/r$, the mean value theorem gives
\[
\langle P(0)u,u\rangle\leq \oint\langle P\varphi,\varphi\rangle|dz|=\oint \langle Pu,u\rangle |dz|+\oint\langle Pu,v\rangle d\overline z+\oint (Pv,u\rangle dz+\oint\langle Pv,v\rangle |dz|.
\]
But this is equivalent to (4.2).

To prove that (ii) or (iii) imply (i), we start by assuming $P\in C^2_{\op} (\Omega,\End^+ V)$.
Then we need to show $P_{z\overline z}\geq P_{\overline z} P^{-1} P_z$.
Let $z_0\in\Omega$ and $z=z_0+\zeta$.
As $\zeta\to 0$ 
\[
P(z)=P(z_0)+P_z(z_0)\zeta+P_{\overline z}(z_0)\overline\zeta+P_{z\overline z}(z_0) |\zeta|^2+\text{ Re }P_{zz}(z_0)\zeta^2+o |\zeta|^2,
\]
whence, as $r\to 0$
\begin{multline*}
\oint P(z) |dz|-\oint P(z) dz \left(\oint P(z) |dz|\right)^{-1} \oint P(z) d\overline z-P(z_0)=\\
 r^2\big(P_{z\overline z} (z_0)-P_{\overline z} (z_0) P(z_0)^{-1} P_z (z_0) + o(1)\big).
\end{multline*}
Therefore (4.4) implies $P_{z\overline z}\geq P_{\overline z} P^{-1} P_z$.

Now take a general $P$. (4.2) implies 
\begin{equation}
\oint P(z)|dz|\geq P(z_0).
\end{equation}
With a smooth $\chi\colon\bC\to [0,\infty)$ supported in a disc $D_\rho(0)$, the convolution $Q=\chi*P$ is defined and $C^\infty_{\weak}=C^\infty_{\op}$ in $\Omega_\rho=\{z\in\Omega\colon\dist (z,\partial\Omega)>\rho\}$, cf.~Proposition 2.3.
If $\chi$ is rotationally symmetric and $\int_{\bC}\chi=1$, then $Q\geq P$ on $\Omega_\rho$ by (4.8).
Clearly, $Q$ inherits the mean value property (4.2) from $P$, as does $t\Id_V+Q$ for any $t>0$.
By what we have just proved, $t\Id_V+Q$ is seminegatively curved.
Further, given a compact $K\subset\Omega_\rho$, $\sup_K \|Q\|$ is dominated by $\sup\| P\|<\infty$, the latter $\sup$ taken over the $\rho$--neighborhood of $K$.
Choose a sequence  of $\chi=\chi_n$ as above, whose supports shrink to $0$, and also $t_n\searrow 0$.
Then  $Q_n=t_n\Id_V+\chi_n*P\to P$ pointwise in the weak operator topology.
Therefore, with any holomorphic $\varphi\colon\Omega\to V$
\[
\langle P\varphi,\varphi\rangle=\lim \langle Q_n\varphi,\varphi\rangle=\inf \langle Q_n\varphi,\varphi\rangle
\]
is subharmonic.
\end{proof}

For a subharmonic function $u$ the integrals $\int_0^{2\pi} u(z_0+r e^{it})dt$ increase with $r$.
This property also generalizes to seminegatively curved metrics and the modified averages in (4.3), (4.4), that we will denote 
\begin{equation}
\frak S (P,z_0,r)=\oint P|dz|-\oint Pdz \left(\oint P|dz|\right)^{-1} \oint P d\overline z.
\end{equation}
The expression makes sense for a general weakly measurable, bounded $P\colon\partial D_r (z_0)\to\End^{\geq 0} V$, assuming $\oint P |dz|$ is invertible.
Failing that, we define
\[
\frak S (P,z_0,r)=\lim_{t\searrow 0}\frak S (P+t\Id_V, z_0,r)
\]
as long as the limit exists.
As we have seen, the limit will exist when $P$ has seminegative curvature.
Thus $\frak S(P,z_0,r)$ is the Schur complement in
\begin{equation}
\frak M (P,z_0,r)=\begin{pmatrix} \oint P|dz|& r\oint Pdz\\  r\oint Pd\overline z& r^2\oint P|dz|\end{pmatrix}.
\end{equation}

\begin{thm}If $P\colon\Omega\to\End^{\geq 0} V$ is seminegatively curved, and $\overline{D_r(z_0)}\subset \Omega$, then
\begin{equation}
\frak S(P,z_0,\rho)\leq\frak S (P,z_0,r)\qquad\text{for }0<\rho\leq r.
\end{equation}
\end{thm}

This generalizes (4.3), which at least for continuous $P$ is obtained in the limit $\rho\to 0$.
The proof requires some preparation.

\begin{lem}Let $V,W$ be Hilbert spaces, $A=A^*$, $A'={A'}^*\in\End\, V$, $B,B'\in\Hom (W,V)$, and $C,C'\in\End^+ W$.
\item{(i)} $(B+B')(C+C')^{-1}(B+B')^*\leq BC^{-1} B^* +B' C'^{-1} {B'}^*$;
\item{(ii)} if
\[
\begin{pmatrix} A&B\\ B^*&C\end{pmatrix}\leq \begin{pmatrix} A'&B'\\ {B'}^*&C'\end{pmatrix}
\]
then $A-BC^{-1} B^*\leq A'-B' C'^{-1}{B'}^*$.
\end{lem}

The result extends to noninvertible $C,C'\ge 0$ if e.g.~$BC^{-1}B^*$ is defined, as above, by $\ds\lim_{t\searrow 0} B(C+t\Id_V)^{-1}B^*$, whenever this limit exists in the strong operator topology.---The first inequality says that the function $(B,C)\mapsto BC^{-1}B^*$ is subadditive (equivalently, convex).

\begin{proof}For finite dimensional $V,W$ statements equivalent to (i) and (ii) are formulated in [HJ, 7.7.P41]. In our generality,
(i)\ was proved in [LR]. It is also straightforward from Lemma 4.3, for
\[
\begin{pmatrix} BC^{-1}B^*+B'{C'}^{-1}{B'}^*&B+B'\\(B+B')^*&C+C'\end{pmatrix}=
\begin{pmatrix}BC^{-1}B^*&B\\B^*&C\end{pmatrix}+\begin{pmatrix}B'{C'}^{-1}{B'}^*&B'\\{B'}^*&C'\end{pmatrix}\ge 0
\]
by Lemma 4.3; and another application of Lemma 4.3 then gives (i). In turn (ii) follows if we introduce 
\[
\begin{pmatrix}\alpha&\beta\\ \beta^*&\gamma\end{pmatrix}=\begin{pmatrix} A'&B'\\ {B'}^*&C'\end{pmatrix}-\begin{pmatrix}A&B\\ B^*&C\end{pmatrix}\geq 0.
\]
We can assume that $\gamma\in\End^+V$, for the general case will then follow by replacing $C'$ by $C'+t\Id$ and letting $t\searrow 0$. By (i)
\[
A'-B' C'^{-1} {B'}^*\geq A+\alpha-BC^{-1} B^*-\beta\gamma^{-1}\beta^*\geq A-BC^{-1} B^*.
\]
\end{proof}

\begin{proof}[Proof of Theorem 4.4]Given $u,v\in V$, consider the holomorphic function $\varphi(z)=u+iv(z-z_0)$, $z\in\Omega$.
Then 
\[
\frac1\rho\int_{\partial D(z_0,\rho)}\langle P\varphi(z),\varphi(z)\rangle |dz|\leq\frac1r\int_{\partial D(z_0,r)}\langle P\varphi(z),\varphi(z)\rangle |dz|
\]
by subharmonicity.
Writing this out in terms of $u,v$ as in the proof of Theorem 4.2, we find that
\[
\frak M (P,z_0,\rho)\leq\frak M(P,z_0,r),
\]
notation as in (4.10).
Hence  Lemma 4.5(ii) and the remark following it imply (4.11).
\end{proof}

An analog of Hadamard's Three Circles Theorem, namely that $\frak S(P,z_0,e^t)$ is convex in $t$, can be proved along the same lines.

Next we turn to metrics $h$ and corresponding $P\colon\Omega\to\End\, V$ with semipositive curvature.
The proper generality for semipositive curvature is not even finite hermitian metrics, but duals of such.
However, not to overburden the discussion, we will only deal with $C^2_{\op}$  operators $P$ that take invertible values.
Then semipositive curvature means
\begin{equation}
P_{\overline z} P^{-1} P_z-P_{z\overline z}\geq 0.
\end{equation}
Our integral characterization of (4.12) will be in terms of 
\begin{equation}
\frak T (P,z_0,r)=\min_H H^* (z_0)^{-1}\frak S (H^* PH, z_0,r) H(z_0)^{-1}
\end{equation}
(when $\overline{D_r(z_0)}\subset\Omega)$, where the minimum is taken over $H\in C_{\op} (\overline{D_r(z_0)}, \GL(V))$ that are holomorphic on $D_r(z_0)$. That there is a minimum is the content of Lemma 4.6 below.
The quantity $\frak T (P,z_0,r)$ is a gauge covariant version of $\frak S(P,z_0,r)$, in the sense that if we transform the gauge---that is, we change the trivialization of $\Omega\times V\to \Omega$---and replace $P$  by $Q=K^*PK$, where $K\colon\Omega\to \GL(V)$ is holomorphic, then
\begin{equation}
\frak T(Q,z_0,r)=K (z_0)^*\frak T (P,z_0,r) K(z_0).
\end{equation}
Curvature (4.12) is also gauge covariant,
\begin{equation}
Q_{\overline z}Q^{-1} Q_z-Q_{z\overline z}=K^* (P_{\overline z} P^{-1} P_z-P_{z\overline z}) K.
\end{equation}
Like $\frak S$, $\frak T(P,z_0,r)$ makes sense when $P$ is defined only on $\partial D_r(z_0)$ and has some mild regularity properties.

\begin{lem}Suppose $0<k <1$ and $P\in C^k (\partial D_r (z_0)$, $\End^+ V$).
Then the minimum in (4.13) is achieved by an $H$ for which $H^*PH=\Id_V$ on $\partial D_r(z_0)$.
Thus $\frak T (P,z_0,r)=H^*(z_0)^{-1} H(z_0)^{-1}$.
\end{lem}

\begin{proof}Assume first that $P\equiv\Id_V$.
For any competing $H$ the curvature of $H^*H$ is $0$ on $D_r(z_0)$.
By Theorem 4.2
\[
H^* (z_0)^{-1}\frak S (H^*H,z_0,r) H(z_0)^{-1}\geq\Id_V,
\]
and equality holds if $H=\Id_V$.

Second, for a general $P$ by Theorems 3.1, 3.7 we can solve the Dirichlet problem
\[
(K^{-1})^* K^{-1}=P\quad\text{ on }\quad\partial D_r (z_0)
\]
for $K\in C^k (\overline{D_r(z_0)}, \GL(V))$ holomorphic on $D_r(z_0)$.
We then pass from $P$ to $K^*PK$ and apply covariance (4.14).
\end{proof}

\begin{thm}A function $P\in C_{\op}^2 (\Omega,\End^+ V)$ has semipositive curvature if and only if for every disc $\overline{D_r(z_0)}\subset\Omega$ 
\begin{equation}
\frak T(P,z_0,r)\leq P(z_0).
\end{equation}
\end{thm}

\begin{proof}Suppose $P$ is semipositively curved, and choose $H\colon\overline{D_r(z_0)}\to \GL(V)$ as in Lemma 4.6.
Since $Q=H^{*-1} H^{-1}$ has zero curvature on $D_r(z_0)$ and $Q=P$ on $\partial D_r(z_0)$, the maximum principle (Theorem 1.3) implies 
%\begin{eqnarray*}
\[
P(z_0)\geq H^* (z_0)^{-1} H(z_0)^{-1}=H^* (z_0)^{-1}\frak S (H^* PH,z_0,r) H(z_0)^{-1}\geq\frak T(P,z_0,r).
\]%\end{eqnarray*}

Conversely, suppose (4.16) holds.
We need to prove $P_{\overline z}P^{-1} P_z\geq P_{z\overline z}$, that we will do at $z_0=0$.
First assume that with some $A=A^*\in\End\, V$
\begin{equation}
P(z)=\Id_V +A|z|^2 +o |z|^2\qquad\text{ as }z\to 0,
\end{equation}
and choose $H=H_r$ again as in Lemma 4.6.
We claim that
\[
\|\Id_V +Ar^2-H_r^* (z)^{-1} H_r (z)^{-1}\|=o(r^2)\quad\text{ as } |z|\leq r\to 0.
\]
Indeed, given $\var>0$, for sufficiently small $r$ and $z\in\partial D_r(0)$
\[
(1-\var r^2)\Id_V +Ar^2\leq H_r^* (z)^{-1} H_r(z)^{-1}\leq (1+\var r^2)\Id_V+Ar^2.
\]
By the maximum principle the same must then hold for $z\in\overline{D_r(0)}$.

But then
\[
P(0)=\Id_V\geq\frak T (P,z_0,r)=H_r^* (0)^{-1} H_r(0)^{-1}=\Id_V+Ar^2+o(r^2),
\]
and $P_{\overline z}(0)P^{-1}(0) P_z(0)-P_{z\overline z}(0)=-A\geq 0$ follows by letting $r\to 0$.

Now for a general $P$ we can choose a holomorphic $K\colon\Omega\to \GL(V)$ so that $P_1=K^*PK$ has Taylor series as in (4.17).
That $P$ has semipositive curvature then follows from the gauge covariance properties (4.14), (4.15) and from the first part of the proof.
\end{proof}

\section{Limits}

Nevertheless, there are limits to how far one can go and generalize results from traditional to noncommutative potential theory. Consider the case of Harnack's inequality: If $\Omega=\{z\in\bC\,:\,|z|<1\}$ and $u:\Omega\to[0,\infty)$ is harmonic, then
\begin{equation}
u(z)\le\frac{1+|z|}{1-|z|} u(0).
\end{equation}
When $u(0)=0$, we recover the maximum---or rather, minimum---principle, $u\equiv 0$. But (5.1) also implies that the maximum principle is {\sl stable}: Knowing how far $u(0)$ is from $\inf u$, we can estimate how far $u$ is from a constant function. This raises the following question of noncommutative potential theory.

Let $\Omega=\{z\in\bC\,:\,|z|<1\}$, and let $P\in C^2_{\op}(\Omega, \End^+V)$ satisfy $R^P=0$. Given that $P(z)\ge\Id$ for all $z\in\Omega$, is it possible to estimate $P(z)$ in terms of $P(0)$, in the form
\[
||P(z)||\le C\big(z,P(0)\big)?
\]

Such an estimate holds and follows from Harnack's inequality (5.1) if $\dim V<\infty$, but not otherwise:

\begin{thm} Suppose $\dim V=\infty$. With $\Omega=\{z\in\bC\,:\,|z|<1\}$ and $z_0\in\Omega\setminus\{0\}$, there is a $T\in\End^+V$ such that
\begin{equation}
\sup\{||P(z_0)||\,:\, P\in C^2_{\op}(\Omega, \End^+ V), \,P\ge\Id,\, R^P=0,\, P(0)=T\}=\infty.
\end{equation}
$T$ can be chosen a multiple of $\Id$.
\end{thm}

This we will derive from a lemma that disproves the noncommutative generalization of Hurwitz's theorem on zeros of holomorphic functions.

\begin{lem} If $\dim V=\infty$, there is a sequence of holomorphic $H_k:\bC\to\GL(V)\subset\End \,V$ with $H_k(0)=\Id$ and $\lim_{k\to\infty}H_k=H:\bC\to\End\, V$ locally uniformly, such that 
\[
\emptyset\neq\{\zeta\in\bC\,:\, H(\zeta)\in\GL(V)\}\neq\bC .
\]
Locally uniform convergence is understood with respect to the norm topology on $\End\, V$.
\end{lem}
\begin{proof} At the bottom of this phenomenon is the fact that the set valued function that associates with an operator $A\in\End\, V$ its spectrum spec$\,A\subset\bC$ is discontinuous, when the space of compact subsets of $\bC$ is endowed with the Hausdorff metric (although the function is upper semicontinuous). A construction showing this is in [Ri, p. 282], attributed to Kakutani. We reproduce this construction, with a minimal change in notation. Consider on $V=l^2$ the weighted shift operator $A$ that maps $x=(x_n)_{n\ge1}\in l^2$ to
\[
\big(0,x_1,\frac{x_2}2,x_3,\frac{x_4}4,x_5,\frac{x_6}2, x_7,\frac{x_8}8,\ldots\big) .
\]
If $\beta(n)$ denotes the highest power of $2$ that divides $n\ge 1$, then
\[
Ax=y,\qquad \text{where }\quad y_n=\begin{cases}0 &\text{if } n=1\\
x_{n-1}/\beta(n-1) & \text{otherwise.}
\end{cases} 
\] 
Further, for $k=1,2,\ldots$ let $A_k\in\End\, V$ be given by
\[
A_kx=y,\qquad \text{where }\quad y_n=\begin{cases}0 &\text{if } 2^k \text{ divides } n-1\\
x_{n-1}/\beta(n-1) & \text{otherwise.}
\end{cases} 
\]
Then $||A-A_k||=2^{-k}$, and $A_k\to A$. Further, $A_k$ is nilpotent, so for all $\zeta\in\bC$
\[
H_k(\zeta)=\Id-\zeta A_k
\]
has an inverse, namely $\sum_{j\le2^k}(\zeta A_k)^j$. However, $H(\zeta)=\lim_k H_k(\zeta)$ is not invertible if $|\zeta|\ge 2$, it is not even onto (while $H(0)=\Id$).

Indeed, suppose $H(\zeta)x=(1,0,0,\ldots)$. This means $x_1=1$ and $x_n=\zeta x_{n-1}/\beta(n-1)$ for $n\ge 2$, i.e., 
\[
x_n=\zeta^{n-1}/\prod_{m<n}\beta(m).
\]
When $n=2^k$, the product is easy to compute. Given $j=0,1,\ldots,k-1$, the equation $\beta(m)=2^j$ has $2^{k-j-1}$ many solutions $1\le m<2^k$. Hence
\[
\prod_{m<n}\beta(m)=\prod_{j=0}^{k-1}2^{j2^{k-j-1}}=2^{2^k\sum_0^{k-1}j2^{-j-1}}<2^n.
\]
Therefore $|x_n|>1/2$ if $|\zeta|\ge2$ and $n$ is a power of $2$, whence $(x_n)_{n\ge 1}\notin l^2$.

This construction for $V=l^2$ has an obvious extension to any infinite dimensional $V$ via a splitting $V=l^2\oplus W$.
\end{proof}

\begin{proof}[Proof of Theorem 5.1] 
By scaling the dependent and independent variables of $H_k, H$ of Lemma 5.2 we can construct a sequence $L_k:\Omega\to\GL(V)$ of holomorphic maps, $||L_k(z)||\le 1$ for $z\in\Omega$, such that $L_k(0)=\var\Id$ with $\var>0$ independent of $k$, and $L_k\to L$ uniformly, but  $L(z_0)$ is not invertible. Then $P_k=L_k^{*-1}L_k^{-1}$ are competitors in (5.2) if $T=\var^{-2}\Id$, and  $\sup_k||P_k(z_0)||=\sup_k||L_k(z_0)^{-1}||^2=\infty$, for otherwise $L(z_0)$ would be invertible by Lemma 3.6.
\end{proof}


\begin{thebibliography}{KMM}
%bibitem[AK]{AK}
\bibitem[ADW]{ADW}S.~Axelrod, S.~Della Pietra, E.~Witten, Geometric quantization of Chern--Simons gauge theory, J.~Diff.~Geo 33 (1991) 787--902.
\bibitem[BK]{BK}R.~Berman, J.~Keller, Bergman geodesics, complex Monge--Amp\`ere equations and geodesics in the space of K\"ahler metrics, Lecture Notes in Math., Vol 2038, Springer, Heidelberg 2012, 283--302.
\bibitem[B]{B}B. Berndtsson, Curvature of vector bundles associated to holomorphic fibrations. Ann. of Math. (2) 169 (2009) 531--160
\bibitem[CS]{CS}R.R.~Coifman, S.~Semmes, Interpolation of Banach spaces, Perron processes, and Yang--Mills, Amer.~Math.~J.~115 (1993) 243--278.
\bibitem[C]{C}J. Conway, A course in functional analysis, 2nd ed., Springer, New York, 1990.
\bibitem[De]{De}A.~Devinatz, The factorization of operator valued functions, Ann.~of Math.~73 (1961) 458--495.
\bibitem[Do]{Do}R.G.~Douglas, On factoring positive operator functions, J.~Math.~and Mechanics, 16 (1966) 119--126.
\bibitem[H]{H}H.~Helson, Lectures on invariant subspaces, Academic Press, New York 1964.
\bibitem[HJ]{HJ}R.A.~Horn, C.R.~Johnson, Matrix analysis 2nd ed., Cambridge University Press, New York, N.Y., 2013.
\bibitem[L1]{L1}L.~Lempert, La m\'etrique de Kobayashi et la repr\'esentation des domaines sur la boule, Bull.~Soc.~Math.~France 109 (1981) 427--474.
\bibitem[L2]{L2}L.~Lempert, A maximum principle for Hermitian (and other) metrics.
Proc. Amer. Math. Soc. 143 (2015) 2193--2200.
\bibitem[L3]{L3}L.~Lempert, Extrapolation, a technique to estimate, arxiv:1507.06216.
\bibitem[L4]{L4}L. Lempert, Analytic cohomology groups of infinite dimensional complex manifolds, J. Math. Anal. and Appl. 445 (2017) 1428--1446. 
\bibitem[LSz]{LSz}L.~Lempert, R.~Sz\H{o}ke, Direct images, fields of Hilbert spaces, and geometric quantization, Comm.~Math.~Phys.~327 (2014) 49--99.
\bibitem[LR]{LR}E.H.~Lieb, M.B.~Ruskai, Some operator inequalities of the Schwarz type, Adv.~in Math.~12 (1974) 269--273.
\bibitem[P]{P}I.~Privalov, Sur les fonctions conjug\'ees, Bull.~Soc.~Math.~France 44 (1916) 100--103.
\bibitem[Ri]{Ri} C.E. Rickart, General theory of Banach algebras.  The University Series in Higher Mathematics, D. van Nostrand, Princeton, N.J.-Toronto-London-New York, 1960
\bibitem[RSz]{RSz}F.~Riesz, B.~Sz.-Nagy, Le\c cons d'analyse fonctionelle 4th ed., Gauthiers--Villars, Paris, 1965.
\bibitem[Rc]{Rc}R.~Rochberg, Interpolation of Banach spaces and negatively curved vector bundles, Pacific J.~of Math.~110 (1984) 355--376.
\bibitem[Rs]{Rs}M.~Rosenblum, Vectorial Toeplitz operators and the Fej\'er--Riesz theorem, J.~Math.~Anal.~Appl.~23 (1968) 139--147.
\bibitem[SzF]{SzF}B.~Sz.-Nagy, C.~Foia\c s, Sur les contractions de l'espace de Hilbert IX.
Factorisation de la fonction caract\'eristique.
Sous--espaces invariants, Acta Sci.~Math.~(Szeged) 25 (1964) 283--316
\bibitem[WM]{WM}N.~Wiener, P.~Masani, The prediction theory of multivariate stochastic processes, Acta Math.~98 (1957) 111--150
\end{thebibliography}
\end{document}